\documentclass[10pt]{article}

\usepackage[tmargin=1in,lmargin=1.5in,rmargin=1.5in,bmargin=1in]{geometry}
\usepackage{amsmath}
\usepackage{amsthm}
\usepackage{amsfonts}
\usepackage{amssymb}
\usepackage{lastpage}
\usepackage{bbm} 
\usepackage{mathrsfs} 
\usepackage[small]{caption}
\usepackage[pdftex]{graphicx}
\usepackage{enumitem} 
\usepackage{capt-of} 

\usepackage{fancyhdr}
\usepackage[colorlinks=true,linkcolor=blue,citecolor=red]{hyperref}
\usepackage{tikz-cd} 

\newcommand{\R}{\mathbb{R}}

\newcommand{\Z}{\mathbb{Z}}

\newcommand{\F}{\mathcal{F}}
\newcommand{\G}{\mathcal{G}}

\DeclareMathOperator{\spanset}{\operatorname{span}} 
\DeclareMathOperator*{\argmin}{\operatorname{arg\,min}} 

\theoremstyle{plain}
\newtheorem{theorem}{Theorem}[section]

\newtheorem{lemma}[theorem]{Lemma}

\theoremstyle{definition}
\newtheorem{definition}[theorem]{Definition}

\theoremstyle{remark}
\newtheorem*{remark}{Remark}
\newtheorem*{remarks}{Remarks}

\numberwithin{equation}{section}

\setlist[enumerate]{font = \normalfont}

\title{Bilinear quadratures for inner products}
\date{\today}
\author{Christopher A. Wong}

\begin{document}

\maketitle

\begin{abstract}
A bilinear quadrature numerically evaluates a continuous bilinear map, such as the $L^2$ inner product, on continuous $f$ and $g$ belonging to known finite-dimensional function spaces. Such maps arise in Galerkin methods for differential and integral equations. The construction of bilinear quadratures over arbitrary domains in $\R^d$ is presented. In one dimension, integration rules of this type include Gaussian quadrature for polynomials and the trapezoidal rule for trigonometric polynomials as special cases. A numerical procedure for constructing bilinear quadratures is developed and validated.
\end{abstract}

\section{Introduction}



Classical quadratures such as Gaussian and trapezoidal rules accurately evaluate continuous linear functionals such as
\[
\int_{\Omega} f(x) w(x) \, dx
\]
for $f$ in a finite-dimensional space of continuous functions. Bilinear quadratures evaluate continuous bilinear forms such as the weighted $L^2$ inner product
\[
\langle f, g \rangle_{L^2} = \int_{\Omega} f(x) g(x) w(x) \, dx
\]
or the weighted $H^1$ inner product
\[
\langle f, g \rangle_{H^1} = \int_{\Omega} \sum_{i,j=1}^d \Big(\frac{\partial f}{\partial x_i} a_{ij}(x) \frac{\partial g}{\partial x_j}\Big) + f(x) g(x) \, dx
\]
on finite-dimensional spaces of continuous functions $f,g$ on $\Omega \subset \R^d$.

$L^2$ inner products compute orthogonal projections onto subspaces, while $H^1$ inner products provide local solutions to elliptic problems, a key ingredient of the finite element method. For example, let $\Omega \subset \R^d$ be a smooth bounded domain, let $L$ be a uniformly elliptic operator, $f \in L^2(\Omega)$, $g \in L^2(\partial \Omega)$, and $\gamma \in L^{\infty}(\partial \Omega)$. Consider the Robin problem
\begin{equation}
\label{eqn.ellipticPDErobin}
\begin{cases}
\text{Find } u \in H^1(\Omega) \text{ satisfying} \\
Lu = f \text{ in } \Omega, \gamma u + \tfrac{\partial u}{\partial n} = g \text{ on } \partial \Omega.
\end{cases}
\end{equation}
When $L = - \Delta$, then if a bilinear form $a: H^1(\Omega) \times H^1(\Omega) \rightarrow \R$ is defined by $a(u,v) = \int_{\Omega} Du \cdot Dv + \int_{\partial \Omega} \gamma uv$, the weak formulation to \eqref{eqn.ellipticPDErobin} seeks $u \in H^1(\Omega)$ satisfying
\begin{equation}
\label{eqn.ellipticPDErobinweak}
a(u,v) = \langle f, v \rangle_{L^2(\Omega)} + \langle g, v \rangle_{L^2(\partial \Omega)} \text{ for all } v \in H^1(\Omega).
\end{equation}
The Galerkin method constructs an approximate solution to \eqref{eqn.ellipticPDErobinweak} by choosing finite-dimensional function spaces $\F_0, \G_0$ and seeking $u_0 \in \F_0$ satisfying
\begin{equation}
\label{eqn.ellipticPDErobinweakgalerkin}
a(u_0,v_0) = \langle f, v_0 \rangle_{L^2(\Omega)} + \langle g, v_0 \rangle_{L^2(\partial \Omega)} \text{ for all } v_0 \in \G_0.
\end{equation}
The linear system \eqref{eqn.ellipticPDErobinweakgalerkin} is solved in a basis, which requires computing a number of $L^2$ inner product integrals. These integrals should be computed both efficiently and accurately.

Efficiency is achieved by using the fewest function evaluations possible. When $d > 1$, the optimal efficiency of a classical quadrature is unknown. For a bilinear quadrature, the minimum number of function evaluations is equal to the dimension of function space being integrated. The inner product of two functions $f,g$ belonging to given finite-dimensional function spaces is computed by the formula
\begin{equation}
\label{eqn.biquadformula1}
\langle f, g \rangle = f(\mathbf{x})^{\ast} W g(\mathbf{y}),
\end{equation}
where $f(\mathbf{x}) \in \R^m$ and $g(\mathbf{y}) \in \R^n$ are evaluations of $f$ and $g$ at sets of points $\mathbf{x}$ and $\mathbf{y}$ in $\Omega$, respectively, and $W$ is a matrix. The rank of the bilinear form is equal to the rank of $W$, hence the minimal number of required function evaluations is equal to the dimension of that function space.


Accuracy is achieved by defining and minimizing integration error. In a bilinear quadrature, this is a nonlinear optimization problem for $\mathbf{x}, \mathbf{y}$, and $W$ in \eqref{eqn.biquadformula1}, and is solved using a Newton method for an appropriate objective function \cite{cheng1,bremer1, xiao1}. In this paper an objective function is developed and demonstrated to yield numerically useful bilinear quadrature rules in a general setting.

Numerical evaluation of inner product integrals has been studied in \cite{boland1,mcgrath, gribble1, bremer1,chen} and as ``bilinear quadrature'' in \cite{luik, knockaert}. This paper borrows some of the framework from these past works but develops and utilizes a different optimization procedure to produce quadrature rules.

\section{Theory}
\subsection{Abstract formulation}
In this section the problem of evaluating a general continuous bilinear form on a pair of Banach spaces is considered. Results are given in great generality so that they apply to any continuous bilinear forms. Later, these results are applied to useful special cases such as the $L^2$ and $H^1$ inner products.

\begin{definition}
Let $\F$ and $\G$ be real Banach spaces. Then a \emph{bilinear quadrature} of order $(m,n)$ on $\F \times \G$ is a bilinear form $Q$ defined by linear maps $L_1: \F \rightarrow \R^m$ and $L_2: \G \rightarrow \R^n$ and a bilinear map $B: \R^m \times \R^n \rightarrow \R$, such that, for each $f \in \F$ and $g \in \G$,
\[ Q(f,g) = B(L_1 f, L_2 g). \]
\end{definition}


\begin{definition}
Let $\F,\G$ be real Banach spaces with a continuous bilinear form $\langle \cdot, \cdot \rangle: \F \times \G \rightarrow \R$. Finite-dimensional subspaces $\F_0 \subset \F$ and $\G_0 \subset \G$ are a \emph{dual pair} if 
\begin{align*}
& \forall f \in \F_0 \setminus \{0\}, \exists g \in \G_0 \text{ such that } \langle f,g \rangle \neq 0, \\
& \forall g \in \G_0 \setminus \{0\}, \exists f \in \F_0 \text{ such that } \langle f,g \rangle \neq 0.
\end{align*}
\end{definition}

If $\F_0, \G_0$ are a dual pair then $\dim(\F_0) = \dim(\G_0)$.


\begin{definition}
Let $\F,\G$ be real Banach spaces with a continuous bilinear form $\langle \cdot, \cdot \rangle: \F \times \G \rightarrow \R$, and let $\F_0 \subset \F$ and $\G_0 \subset \G$ be a dual pair. A bilinear quadrature $Q$ on $\F \times \G$ is \emph{exact with respect to} $\F_0 \times \G_0$ if
\[ \langle f,g \rangle = Q(f,g) \text{ for every } f \in \F_0, g \in \G_0.\]
\end{definition}

Such a bilinear quadrature evaluates the bilinear form on $\F_0 \times \G_0$ exactly. We have the diagram
\[ 
\begin{tikzcd}
\F_0 \times \G_0 \arrow{rd}[swap]{\langle \cdot, \cdot \rangle} \arrow{r}{(L_1, L_2)} &  \R^m \times \R^n \arrow{d}{B(\cdot, \cdot)} \\
 & \R \\
\end{tikzcd}
\]
If the parent spaces $\F, \G$ are implied, we will abuse notation by referring to an exact bilinear quadrature on $\F_0 \times \G_0$.

\begin{remark}
If $\F,\G$ are infinite-dimensional and $Q$ is exact on $\F_0 \times \G_0$, then
\[
\sup_{f \in \F, g \in \G} |Q(f,g) -\langle f,g \rangle| = \infty,
\]
so a bilinear quadrature can only be accurate on finite-dimensional subspaces.
\end{remark}

\begin{lemma}
\label{lemma.quadexist}
Let $\F_0 \subset \F$ and $\G_0 \subset \G$ be a dual pair, and let $L_1: \F \rightarrow \R^m$, $L_2: \G \rightarrow \R^n$ be linear. Then there exists bilinear $B: \R^m \times \R^n \rightarrow \R$ such that the map $Q(f,g) = B(L_1 f,L_2 g)$ is an exact quadrature on $\F_0 \times \G_0$ if and only if $\left. L_1 \right|_{\F_0}$ and $\left. L_2 \right|_{\G_0}$ are both injective.
\end{lemma}

\begin{proof}
Suppose $B$ exists. If $f, \tilde{f} \in \F_0$ are distinct then there exists $g \in \G_0$ such that
\[
0 \neq \langle f - \tilde{f}, g \rangle = Q(f - \tilde{f}, g) = B(L_1(f - \tilde{f}), L_2 g),
\]
so $L_1 f \neq L_1 \tilde{f}$ and $\left. L_1 \right|_{\F_0}$ is injective. Similarly for $\left. L_2 \right|_{\G_0}$.

Suppose $\left. L_1 \right|_{\F_0}$ and $\left. L_2 \right|_{\G_0}$ are injective. Their Moore-Penrose pseudoinverses $(\left. L_1 \right|_{\F_0})^{+}$ and $(\left. L_2 \right|_{\G_0})^{+}$ left-invert $L_1$ and $L_2$, respectively. Define a bilinear map on $\R^m \times \R^n$ by
\[ B(x,y) = \left\langle (\left. L_1 \right|_{\F_0})^{+} x, (\left. L_2 \right|_{\G_0})^{+} y \right\rangle. \]
Then, for all $f \in \F_0, g \in \G_0$,
\begin{align*}
B(L_1 f, L_2 g) & = \left\langle (\left. L_1 \right|_{\F_0})^{+} \left. L_1\right|_{\F_0} f, (\left. L_2 \right|_{\G_0})^{+} \left.L_2\right|_{\G_0}g \right\rangle \\
& = \langle f,g \rangle.
\end{align*}
\end{proof}
From Lemma~\ref{lemma.quadexist} a necessary condition for an exact bilinear quadrature is that $m \geq \dim{\F_0}, n \geq \dim{\G_0}$. Minimal order is achieved when $m = \dim{\F_0}, n = \dim{\G_0}$ and $B(x,y)$ is uniquely given by
\[
B(x,y) = \left\langle (\left. L_1 \right|_{\F_0})^{-1} x , ( \left. L_2 \right|_{\G_0})^{-1} y \right \rangle.
\]

Exact bilinear quadratures are not unique, as there are many possible linear maps $L_1, L_2$. Furthermore, $B$ may not be unique, since if $n > \dim(\F_0)$, then $\left. L_1 \right|_{\F_0}$ has infinitely many left inverses. Therefore, a method is needed to choose among the infinitely many bilinear quadratures. One metric of quality is that, in addition to its exactness on $\F_0 \times \G_0$, the bilinear quadrature also approximates $\langle f,g \rangle$ for some set of $g$'s outside of $\G_0$.

\begin{definition}
Let $\F_0 \subset \F$ and $\G_0 \subset \G$ be a dual pair, and let $\G_1 \subset \G$ be another finite-dimensional subspace such that
\[\G_1 \subset \F_0^{\perp}:= \{g \in G: \langle f, g \rangle = 0 \text{ for all } f \in \F_0\}.
\]
Let $\mathcal{Q}$ be a set of bilinear quadratures exact on $\F_0 \times \G_0$. Then $Q \in \mathcal{Q}$ is called \emph{minimal on} $\G_1$ if
\begin{equation}
\label{eqn.minquad}
Q = \argmin_{\tilde{Q} \in \mathcal{Q}} \sigma(\tilde{Q}; \F_0, \G_1),
\end{equation}
where
\[
\sigma(\tilde{Q}; \F_0, \G_1) :=\max_{ \substack{0 \neq g \in \G_1 \\ 0 \neq f \in \F_0}} \frac{ | \tilde{Q}(f,g)|}{\|f\|_{\F} \|g\|_{\G}}.
\]
\end{definition}

If $Q$ is minimal on $\G_1$, then it approximates the pairing of $\F_0$ and $\G_0 \oplus \G_1$. Precisely, if $f \in \F_0, g \in \G_0 \oplus \G_1$, and we write $g = g_0 + g_1$ with $g_i \in \G_i$, then
\begin{equation}
\label{eqn.minquaderror}
|Q(f,g) - \langle f,g \rangle| = |Q(f,g_1)| \le \sigma(Q; \F_0, \G_1) \|f\|_{\F} \|g_1\|_{\G}.
\end{equation}
Thus, minimizing $\sigma(Q; \F_0, \G_1)$ will improve the approximation.


One important special case for bilinear quadratures is the symmetric case, which is when $F = G$ is an inner product space. In this case, a bilinear quadrature computes an orthogonal projection.

\begin{definition}
Let $\F_0$ and $\G_0$ be a dual pair in an inner product space. Let $\{f_i\}$ be an orthonormal basis for $\F_0$. Given a bilinear quadrature $Q$ exact on $\F_0 \times \G_0$, the \emph{approximate orthogonal projection} onto $\F_0$ arising from $Q$ is the linear map $P_Q$ given by
\[
P_Q(g) = \sum_i Q(f_i, g) f_i.
\]
\end{definition}
An error estimate for orthogonal projections similar to \eqref{eqn.minquaderror} is given later in Theorem~\ref{thm.projectionerror}.

\subsection{Integral formulation}

In this section, the bilinear quadrature framework is applied to the evaluation of Sobolev inner products on function spaces. Let $\Omega \subset \R^d$ be a bounded domain, and let $\F = \G = C^r(\overline{\Omega})$, $r$ a non-negative integer, equipped with a Sobolev inner product

\[
\langle f,g \rangle_{H^s} = \sum_{|\alpha| \le s} \langle D^{\alpha}f,  D^{\alpha} g \rangle_{L^2(\Omega)}
\]
for $s \leq r$.

Choose a dual pair $\F_0, \G_0$ in $\F$. Exactness on $\F_0 \times \G_0$ requires linear maps $L_1: C^r(\overline{\Omega}) \rightarrow \R^m$, $L_2: C^r(\overline{\Omega}) \rightarrow \R^n$, and bilinear form $B: \R^m \times \R^n \rightarrow \R$ so that for every $f \in \F_0, g \in \G_0$, $B(L_1 f, L_2 g) = \langle f, g \rangle_{H^s}$.

Appropriate linear maps $L_1, L_2$ are pointwise evaluations at particular points in $\Omega$. Thus, for the points $\mathbf{x} = (x_1, \ldots, x_m) \in \Omega^m$, $\mathbf{y} = (y_1, \ldots, y_n) \in \Omega^n$, define
\[
L_1f := f(\mathbf{x}) = \begin{bmatrix} f(x_1) \\ \vdots \\ f(x_m) \end{bmatrix},\quad
L_2g := g(\mathbf{y}) = \begin{bmatrix} g(y_1) \\ \vdots \\ g(y_n) \end{bmatrix}.
\]
Given bases $\beta = \{f_1, \ldots, f_k\}$ for $\F_0$ and $\{g_1, \ldots, g_k\}$ for $\G_0$, let $M \in \R^{k \times k}$ be the Gram matrix with entries
\[
M_{ij} = \langle f_i, g_j \rangle_{H^s}.
\]
Since $\F_0, \G_0$ are a dual pair, $M$ is invertible. Define matrix functions
\[
F(\mathbf{x}) := \begin{bmatrix} \F_1(x_1) & \ldots & f_k(x_1) \\ \vdots & & \vdots \\ \F_1(x_m) & \ldots & f_k(x_m) \end{bmatrix}, \quad G(\mathbf{y}) := \begin{bmatrix} \G_1(y_1) & \ldots & g_k(y_1) \\ \vdots & & \vdots \\ \G_1(y_n) & \ldots & g_k(y_n) \end{bmatrix}.
\]
To make $L_1$ and $L_2$ are injective, choose $\mathbf{x}, \mathbf{y}$, such that $F(\mathbf{x})$ and $G(\mathbf{y})$ have full column rank. If $B(v,w) = v^{\ast} W w$ for all $v,w$ for an $m \times n$ matrix $W$, then the bilinear quadrature is exact if and only if 
\begin{equation}
\label{eqn.Wmatrixconstraint}
F(\mathbf{x})^{\ast} W G(\mathbf{y}) = M.
\end{equation}
Therefore a bilinear quadrature rule
\begin{equation}
\label{eqn.innerproductquad}
Q(f,g) = f(\mathbf{x})^{\ast} W g(\mathbf{y})
\end{equation}
evaluates $\langle f,g \rangle_{H^s}$ exactly for any $f \in \F_0, g \in \G_0$. The corresponding approximate orthogonal projection onto $\F_0$ is
\[
P_Q(g) = \sum_{i=1}^k \left[f_i(\mathbf{x})^{\ast} W g(\mathbf{y}) \right] f_i.
\]
In the basis $\beta$, the approximate projection is computed by
\begin{equation}
\label{eqn.approximateprojection}
[P_Q(g)]_{\beta} = F(\mathbf{x})^{\ast} W g(\mathbf{y}) \in \R^k.
\end{equation}


Good values for the matrix $W$ and evaluation points $\mathbf{x},\mathbf{y}$ must be determined. Without loss of generality, suppose that the bases $\{f_i\}$ and $\{ g_j\}$ are $H^s$-orthonormal in $C^r(\overline{\Omega})$. Select finite-dimensional $\G_1 \subset C^r(\overline{\Omega})$ for the minimization \eqref{eqn.minquad} and define the feasible set $\mathcal{Q}$ to be all quadratures of the form \eqref{eqn.innerproductquad} satisfying \eqref{eqn.Wmatrixconstraint}. If $\{\gamma_1, \ldots, \gamma_p\}$ is an orthonormal basis for $\G_1$, define
\[
\Gamma(\mathbf{x}) := \begin{bmatrix} \gamma_1(x_1) & \ldots & \gamma_p(x_1) \\ \vdots & & \vdots \\ \gamma_1(x_n) & \ldots & \gamma_p(x_n) \end{bmatrix} \in \R^{n \times p}.
\]
Then \eqref{eqn.minquad} can be reformulated as
\begin{align}
\notag \min_{Q \in \mathcal{Q}} \sigma(Q; \F_0, \G_1) & = \min_{Q \in \mathcal{Q}} \max_{\substack{g \in \G_1, \|g\|_G = 1 \\ f \in \F_0, \|f\|_{F} =1 }} | Q(f,g)| \\ 
\notag & = \min_{\mathbf{x}, \mathbf{y}, W} \max_{ \substack{a,b \in \R^k \\ \|a\|_2 = \|b\|_2 = 1}} | b^{\ast} F(\mathbf{x})^{\ast} W \Gamma(\mathbf{y}) a | \\ 
\label{eqn.mininnerproduct} & = \min_{\mathbf{x}, \mathbf{y}, W}  \sigma_1\left( F(\mathbf{x})^{\ast} W \Gamma(\mathbf{y}) \right) \text{ subject to } F(\mathbf{x})^{\ast} W G(\mathbf{y}) = M,
\end{align}
where $\sigma_1(A)$ is the leading singular value of a matrix $A$. Minimization \eqref{eqn.mininnerproduct} is independent of $\mathbf{x}$, since by \eqref{eqn.Wmatrixconstraint} $F(\mathbf{x})^{\ast}W = M L$, where $L$ is a left inverse of $G(\mathbf{y})$. Therefore $\mathbf{x}$ is chosen by performing a similar minimization on the left, setting an orthonormal basis $\{\lambda_i\}$ for a space $\F_1 \subset \G_0^{\perp}$, defining the corresponding matrix function $\Lambda(\mathbf{x})$, and minimizing
\begin{equation}
\label{eqn.mininnerproduct2}
\min_{\mathbf{x}, \mathbf{y}, W}  \sigma_1\left( \Lambda(\mathbf{x})^{\ast} W G(\mathbf{y}) \right) \text{ subject to } F(\mathbf{x})^{\ast} W G(\mathbf{y}) = M,
\end{equation}
where similarly the dependence of \eqref{eqn.mininnerproduct2} on $\mathbf{y}$ may be dropped since $W G(\mathbf{y})$ is equal to $L^{\ast} M$, where $L$ is a left inverse of $F(\mathbf{x})$.

In the symmetric case $\F_0 = \G_0$, $M = I$, $\F_1 = \G_1$, and $m = n = k$ with $\mathbf{x} = \mathbf{y}$, minimizations \eqref{eqn.mininnerproduct} and \eqref{eqn.mininnerproduct2} are equivalent and simplify to
\begin{equation}
\label{eqn.mininnerproductsymmetric}
\min_{\mathbf{x}} \sigma_1(F(\mathbf{x})^{-1} \Gamma(\mathbf{x})).
\end{equation}
In subsequent sections special attention is given to the symmetric case because it is used for evaluating orthogonal projections.

\subsection{Error estimates}

In this section, upper bounds on several error quantities in computing an approximate orthogonal projection of the form \eqref{eqn.approximateprojection} are estimated.

\begin{theorem}[Euclidean norm error estimate]
\label{thm.projectionerror}
Let $\F_0, \G_0$ be a dual pair in an inner product space $\F$ and $Q$ a bilinear quadrature of the form \eqref{eqn.innerproductquad} that is exact on $\F_0 \times \G_0$. Let $P_Q$ be the approximate orthogonal projection onto $\F_0$ arising from $Q$ with coordinate representation \eqref{eqn.approximateprojection}. If $P$ is the exact orthogonal projection operator onto $\F_0$, $\G_1 \subset \F_0^{\perp}$, and $g = g_0 + g_1 \in \G_0 \oplus \G_1$ such that $g_i \in \G_i$,
\begin{equation}
\label{eqn.projectionerror}
\| [P_Q(g) - P(g)]_{\beta} \|_2 \le \sigma(Q; \F_0, \G_1) \|g_1\|,
\end{equation}
where $\| \cdot \|_2$ is the Euclidean norm.
\end{theorem}

\begin{proof}
This is essentially the same as \eqref{eqn.minquad}. Since $\langle f_i, g \rangle = Q( f_i, g_0)$, then
\begin{align*}
\| [P_Q(g) - P(g)]_{\beta} \|_2 & = \left(\sum_{i=1}^k \left| Q(f_i,g) - \langle f_i, g \rangle\right|^2\right)^{1/2} \\
& = \left(\sum_{i=1}^k | Q(f_i, g) - Q(f_i,g_0) |^2\right)^{1/2} \\
& = \left(\sum_{i=1}^k | Q(f_i,g_1) |^2\right)^{1/2} \\
& = \max_{\alpha \neq 0} \frac{1}{\| \alpha\|_2} \sum_{i=1}^k \alpha_i Q(f_i, g_1),
\end{align*}
where $\alpha = (\alpha_i) \in \R^k$. Each $f \in \F_0$ can be written as $f = \sum_{i} \alpha_i f_i$, so
\begin{align*}
\max_{\alpha \neq 0} \frac{1}{\| \alpha\|_2} \sum_{i=1}^k \alpha_i Q(f_i, g_1)& = \max_{0 \neq f \in \F_0} \frac{Q(f, g_1)}{\|f\|} \\
& \le \sigma(Q; \F_0, \G_1) \|g_1\|.
\end{align*}
\end{proof}

Theorem~\ref{thm.projectionerror} provides an error bound for an approximate orthogonal projection when the projected function $g$ is in $\G_0 \oplus \G_1$. If $\F_0$ is a space of polynomials, then it is also useful to obtain an error estimate that depends on the regularity of $g$.

\begin{theorem}[Uniform norm error estimates for polynomials]
Let $\F = C(\overline{\Omega})$ with $\Omega \subset \R^d$ a bounded, convex domain, equipped with the $L^2$ inner product. Let $\F_0$ be the set of multivariate polynomials of degree at most $n$ with an orthonormal basis $\beta = \{ f_i\}$, let $P: \F \rightarrow \F$ be the orthogonal projection onto $\F_0$, and suppose $P_Q$ is an approximate orthogonal projection onto $\F_0$ with coordinate representation \eqref{eqn.approximateprojection}. There exist a constant $C > 0$ such that for every $g \in C^{n+1}(\overline{\Omega})$, then 
\begin{equation}
\label{eqn.inftyinftyestimate}
\| [Pg - P_Q g]_{\beta} \|_{\infty} \le C \|D^{n+1} g\|_{L^{\infty}},
\end{equation}
where
\[
\|D^{n+1} g\|_{L^{\infty}} := \sum_{|\alpha| = n+1} \max_{x \in \Omega} |D^{\alpha}g(x)|.
\]
\end{theorem}

\begin{proof}
Using \eqref{eqn.approximateprojection} and the exactness of $P_Q$ on $\F_0$, then writing $g = Pg + (I - P)g = g_0 + g_1$, we have
\begin{align*}
\| [Pg - P_Q g]_{\beta} \|_{\infty} & = \| F^{\ast} W g_1(\mathbf{x})\|_{\infty} \\
& \le \| F^{\ast} W\|_{\infty \rightarrow \infty} \|(I - P)g\|_{C^0} \\
& \le \| F^{\ast} W\|_{\infty \rightarrow \infty} \|(I - P)(g - q)\|_{C^0},
\end{align*}
where $q$ is any element of $\F_0$ and $\| \cdot \|_{\infty \rightarrow \infty}$ is the induced matrix $\infty$ norm. Then
\begin{align*}
\| [Pg - P_Q g]_{\beta} \|_{\infty} & \le \| F^{\ast} W\|_{\infty \rightarrow \infty} (1 + \|P\|_{C^0 \rightarrow C^0}) \|g - q\|_{C^0},
\end{align*}
where the $C^0$ operator norm of $P$ is given by
\[
\|P\|_{C^0 \rightarrow C^0} = \max_{x \in \Omega} \int_{\Omega} \Big| \sum_i f_i(t) f_i(x) \Big| \, dt.
\]
By the Deny-Lions/Bramble-Hilbert lemma \cite{ern}, for all $g \in C^{n+1}(\Omega)$ there exists a constant $C_{BH} > 0$ (dependent on $n$ and $\Omega$) such that
\[
\inf_{q \in \F_0} \| g - q\|_{C^0} \le C_{BH} \|D^{n+1}g\|_{L^{\infty}},
\]
which combined with the previous inequality yields the desired result with 
\[
C = \| F^{\ast} W\|_{\infty \rightarrow \infty} (1 + \|P\|_{C^0 \rightarrow C^0}) C_{BH}.
\]

\end{proof}

In the presence of round-off error in function evaluation, the conditioning of an approximate orthogonal projection is also important to quantify.

\begin{theorem}
\label{thm.roundofferror}
Let $P_Q$ be an approximate orthogonal projection of the form \eqref{eqn.approximateprojection}. If $\delta g(\mathbf{y})$ is the absolute error in computing $g(\mathbf{y})$ and $\delta P_Q(g)$ is the resulting projection absolute error, then with respect to a vector norm $\| \cdot \|$,
\[
\frac{ \| [\delta P_Q(g)]_{\beta}\|}{ \|[P_Q(g)]_{\beta}\|} \le \kappa \frac{ \| \delta g(\mathbf{y})\|}{ \|g(\mathbf{y})\|},
\]
where $\kappa = \kappa(F^{\ast}(\mathbf{x})W)$ is the matrix condition number with respect to $\| \cdot \|$.
\end{theorem}

\subsection{Classical and bilinear quadratures on univariate polynomials}

In this section we review Gaussian quadratures and show they are a special case of a bilinear quadrature in one dimension. We then propose a way to generalize to quadratures evaluating inner products of polynomials on multidimensional domains.

\begin{definition}
Let $\Omega \subset \R^d$ be a connected domain. A \emph{classical quadrature} $q$ of order $n$ on $\Omega$ is a linear functional defined by a set $\mathbf{x} = (x_1, \ldots, x_n)$, $x_i \in \Omega$, called the \emph{nodes}, and a vector $\mathbf{w} \in \R^n$, whose components are called the \emph{weights}, such that for any $f \in C(\Omega)$,
\[
q(f) = \mathbf{w}^{\ast} f(\mathbf{x}) = \sum_{i=1}^n w_i f(x_i).
\]
Furthermore, if $\F_0$ is a subspace of $C(\Omega)$ and $\mu$ is a Borel measure, then $q$ is said to be exact on $\F_0$ if 
\[
q(f) = \int_{\Omega} f \, d\mu
\]
for all $f \in \F_0$.
\end{definition}

Let $\mathbb{P}_n$ be the space of univariate polynomials of degree up to $n$, $I$ an open interval, and $\mu$ a finite absolutely continuous Borel measure on $I$. 

\begin{definition}
Suppose $\mathbb{P}_{2n-1}$ is $\mu$-integrable on $I$. Then a \emph{Gaussian quadrature} of order $n$ on $I$ is a classical quadrature of order $n$ on $I$ that is exact on $\mathbb{P}_{2n-1}$ with respect to $\mu$. 
\end{definition}

The advantages and disadvantages of the theory of quadratures for polynomials are rooted in existence and uniqueness result for Gaussian quadratures.

\begin{theorem}
\label{theorem.classicalgauss}
Suppose $\mathbb{P}_{2n-1}$ is $\mu$-integrable on $I$, and let $\{\phi_k\}$ denote any set of $L^2(I,\mu)$-orthonormal polynomials such that $\deg(\phi_k) = k$. Then the following sets are equal:
\begin{enumerate}
\item The zeros of $\phi_n$.
\item The eigenvalues of the symmetric bilinear form on $\mathbb{P}_{n-1}$ given by
\[ B(f,g) := \int_I x f(x) g(x) \, d\mu = \left\langle x f(x), g(x) \right\rangle_{L^2(I,\mu)}. \]
\item The nodes $\{x_i\}$ of a Gaussian quadrature of order $n$ on $I$.
\end{enumerate}
\end{theorem}
\begin{proof}

$(1 \Longleftrightarrow 2)$ Since $B(\cdot,\cdot)$ is symmetric it is diagonalizable with $n$ real eigenvalues $\{\lambda_i\}$. If a polynomial $\psi_i(x)$ is an eigenvector for $\lambda_i$, then for $0 \le j \le n-1$,
\[ \langle x \psi_i(x), \phi_j(x) \rangle = \lambda_i \langle \psi_i(x), \phi_j(x) \rangle \Longrightarrow \langle (x - \lambda_i) \psi_i(x), \phi_j(x) \rangle = 0. \]
Since $(x - \lambda_i) \psi_i(x) \in \mathbb{P}_{n}$ for each $i$, and the only polynomials in $\mathbb{P}_{n}$ that are orthogonal to each of $\phi_0, \ldots, \phi_{n-1}$ are multiples of $\phi_n$, then each $(x - \lambda_i)$ is a factor of $\phi_n(x)$. Thus $\phi_n(x)$ is a multiple of $(x - \lambda_1) \ldots (x - \lambda_n)$ and its zeros are the eigenvalues of $B(\cdot,\cdot)$. \\

$(2 \Longleftrightarrow 3)$ Suppose a Gaussian quadrature with weights $\{w_i\}$ and nodes $\{x_i\}$ exists. With respect to the basis of orthonormal polynomials $\{\phi_k\}$, the bilinear form $B(\cdot,\cdot)$ has a symmetric matrix represention $B$ with entries given by
\[ B_{ij} = \int_I x \phi_i(x) \phi_j(x) \, d\mu = \sum_{k=1}^n w_k x_k \phi_i(x_k) \phi_j(x_k). \]
If
\[
u_k = \begin{bmatrix} \sqrt{w_k} \phi_0(x_k) \\ \vdots \\ \sqrt{w_k} \phi_{n-1}(x_k) \end{bmatrix}, X = \begin{bmatrix} x_1 & & \\ & \ddots & \\ & & x_n \end{bmatrix},
\]
then
\begin{equation}
\label{eqn.multxmatrix}
B = \sum_{k=1}^{n} x_k u_k u_k^{\ast} = U X U^{\ast}.
\end{equation}
Since $\delta_{ij} = \sum_{k=1}^n w_k \phi_i(x_k) \phi_j(x_k)$, then $I = U U^{\ast}$ and $U$ is a unitary matrix. Then \eqref{eqn.multxmatrix} is the unitary diagonalization of the symmetric matrix $B$ with eigenvalues given by the $x_k$'s.
\end{proof}

\begin{remarks}
Theorem \ref{theorem.classicalgauss} shows that if a Gaussian quadrature of order $n$ exists, its nodes are the zeros of $\phi_n$. The existence proof is completed by showing the weights exist and satisfy
\[ 1/w_i = \sum_{j=0}^n (\phi_j(x_i))^2.\]
Therefore, taking the square root $\sqrt{w_k}$ is legitimate \cite{davis}. Theorem \ref{theorem.classicalgauss} also provides an efficient method to construct these quadratures. The matrix $B$ in \eqref{eqn.multxmatrix} is tridiagonal, so its eigenvalues can be calculated quickly, even for very large $n$ \cite{golub}.
\end{remarks}

Gaussian quadrature is optimal for integrating polynomials on an interval, but does not extend readily to higher-dimensional domains. The zeros of a multivariate polynomial are generally not isolated (consider $f(x,y) = xy$) so they cannot all be used as nodes of a classical quadrature. Additionally, the connection between nodes and eigenvalues no longer holds since the eigenvalues are only scalars (the connection extends to two-dimensional domains with complex eigenvalues \cite{vioreanu}).

Bilinear quadratures make sense in any dimension, yet contain Gaussian quadrature as a special case. Consider a classical quadrature with nodes $\mathbf{x} = \{x_i\}$ and weights $\mathbf{w} = \{w_i\}$. If a function $h(x) = f(x) g(x)$ with $f,g$ belonging to function spaces $\F,\G$ respectively, then the classical quadrature $q$ evaluated on $h$ is the same as a bilinear quadrature $Q$ on $f,g$ given by
\[
Q(f,g) = f(\mathbf{x})^{\ast} \begin{bmatrix} w_1 & & \\ & \ddots & \\ & & w_n \end{bmatrix} g(\mathbf{x}) = \mathbf{w}^{\ast} h(\mathbf{x}) = q(h).
\]
The matrix $W$ is diagonal with entries given by the weights of the classical quadrature. Thus for a general bilinear quadrature of the form \eqref{eqn.innerproductquad} the entries of $W$ can be viewed as analogues of the weights.

\begin{theorem}
The nodes of a Gaussian quadrature of order $n$ are the same as the points $\mathbf{x}$ in the unique bilinear quadrature of order $(n,n)$ on $\mathbb{P}_{n-1} \times \mathbb{P}_{n-1}$ that is minimal on $\spanset\{\phi_n\}$, where $\phi_n$ is the orthonormal polynomial of degree $n$. Furthermore, the matrix $W$ in \eqref{eqn.innerproductquad} is a diagonal matrix whose diagonal entries are the weights of the Gaussian quadrature.
\end{theorem}

\begin{proof}
Let $\phi_0, \ldots, \phi_{n-1}$ be the orthonormal polynomials up to degree $n-1$ such that $\deg{\phi_k} = k$, $\mathbf{x} = (x_1, \ldots, x_n) \in \Omega^n$, and define
\[
\Phi(\mathbf{x}) = \begin{bmatrix} \phi_0(x_1) & \ldots & \phi_{n-1}(x_1) \\ \vdots & & \vdots \\ \phi_0(x_n) & \ldots & \phi_{n-1}(x_n) \end{bmatrix}.
\]
Then the minimization problem \eqref{eqn.mininnerproductsymmetric} becomes
\[
\min_{\mathbf{x} \in \Omega^n} \sigma_1\left( \Phi(\mathbf{x})^{-1} \begin{bmatrix} \phi_n(x_1) \\ \vdots \\ \phi_n(x_n) \end{bmatrix} \right).
\]
This is uniquely minimized (up to reordering of the $x_i$'s) when $\mathbf{x}$ is the set of zeros of $\phi_n$ in which case $\sigma_1 = 0$. The corresponding bilinear quadrature $Q$ exactly evaluates products where one polynomial has degree $n-1$ and the other has degree $n$. Then $Q$ has the same evaluation points as a bilinear quadrature formed from the Gaussian quadrature. Since $W$ is unique, then it must be equal to the diagonal matrix with entries given by the weights of the Gaussian quadrature.
\end{proof}

The above result suggests that a good way to accurately compute inner products of polynomials on a multidimensional domain is to utilize a symmetric bilinear quadrature that is exact on $\mathbb{P}_n \times \mathbb{P}_n$ and minimal on $\mathbb{P}_{n+1} \cap \mathbb{P}_n^{\perp}$. Just as Gaussian quadratures accurately integrate nearly polynomial functions accurately, bilinear quadratures constructed in the above manner are expected to evaluate inner products of nearly polynomial functions accurately.  Numerical results for these quadrature are shown in section~\ref{sec.computation}.

\subsection{Classical and bilinear quadratures on trigonometric polynomials}

For the space of trigonometric polynomials
\[ T_{n-1} = \spanset\{1, \sin{x}, \ldots, \sin{(n-1)x}, \cos{x}, \ldots, \cos{(n-1)x}\}, \]
it is known that the $(n+1)$-point trapezoidal rule
\[ \mathrm{Tra}(p) := \frac{\pi}{n} p(0) + \frac{\pi}{n} p(2\pi) + \frac{2\pi}{n} \sum_{j=1}^{n-1} p\left( \frac{2 \pi j}{n} \right)\]
is exact for integrating all $p \in T_{n-1}$ over the interval $[0,2\pi]$. Since $T_{n-1}$ is a rotationally-invariant function space on the circle $\R/2\pi\Z$, the trapezoidal rule yields a family of $n$-point classical quadratures for $T_{n-1}$ given by
\begin{equation}
\label{eqn.trigquad}
\int_0^{2 \pi} p(x) \, dx = \frac{2 \pi}{n} \sum_{j=0}^{n-1} p(x_j) \quad \text{for all } p \in T_{n-1}, \quad x_{j+1} - x_j = \frac{2\pi}{n}.
\end{equation}
When $n$ is odd, the above trapezoidal rule quadrature is a special case of a bilinear quadrature:
\begin{theorem}
\label{thm.trigquad}
Let $n > 0$ be an odd integer. Then the set of classical quadratures on $T_{n-1}$ in \eqref{eqn.trigquad} are equivalent to symmetric bilinear quadratures of order $(n,n)$ on $T_{(n-1)/2} \times T_{(n-1)/2}$ that are minimal on $\spanset\{\sin{nx},\cos{nx}\}$.
\end{theorem}

\begin{proof}
Set $n = 2k+1$. Since $T_{n-1}$ is rotationally invariant on the circle, then if $Q$ is a symmetric bilinear quadrature on $T_k \times T_k$ of the form \eqref{eqn.innerproductquad}, $\sigma(Q)$ is invariant under rotations of the evaluation points $\mathbf{x} = (x_1, \ldots, x_n)$. Therefore without loss of generality $x_1 = 0, x_j \in [0, 2\pi)$. Define
\[ F(\mathbf{x}) = \frac{1}{\sqrt{2\pi}}\begin{bmatrix} 1 & \ldots & 1 \\ e^{-ikx_2} & \ldots & e^{ikx_2} \\ \vdots & & \vdots \\ e^{-ikx_n} & \ldots & e^{ikx_n} \end{bmatrix}, \Gamma(\mathbf{x}) = \frac{1}{\sqrt{2\pi}}\begin{bmatrix} 1 & 1 \\ e^{-i(k+1)x_2} & e^{i(k+1)x_2} \\ \vdots & \vdots \\ e^{-i(k+1)x_n} & e^{i(k+1)x_n} \end{bmatrix}, \]
and it suffices to prove that choosing $x_j = 2\pi j/n$ solves the minimization problem \eqref{eqn.minquad}.

If $x_j = 2 \pi j/n$, then the first column of $F(\mathbf{x})$ is the second column of $\Gamma(\mathbf{x})$, and the last column of $F(\mathbf{x})$ is the first column of $\Gamma(\mathbf{x})$. Therefore, in this case, $F(\mathbf{x})^{-1} \Gamma(\mathbf{x}) = \begin{bmatrix} e_n & e_1 \end{bmatrix}$, where $e_j$ is the $j$-th standard coordinate vector, hence $\sigma_1(F(\mathbf{x})^{-1} \Gamma(\mathbf{x})) = 1$.

We claim that for any choice of nodes $\mathbf{x} \in [0,2\pi)^n$ with $x_1 = 0$,
\begin{equation}
\label{eqn.trigquadsigmamin}
\sigma_1(F(\mathbf{x})^{-1} \Gamma(\mathbf{x})) \geq 1.
\end{equation}
If \eqref{eqn.trigquadsigmamin} is established, then setting $x_j = 2\pi (j-1)/n$ yields a minimal quadrature in $\mathcal{Q}$. To show this, let $\mathbf{u}$ be the first column of $F(\mathbf{x})^{-1} \Gamma(\mathbf{x})$. We will show that $\|\mathbf{u}\|_2 \geq 1$, from which \eqref{eqn.trigquadsigmamin} follows. The column $\mathbf{u}$ satisfies the equation
\[ F(\mathbf{x}) \mathbf{u} = \frac{1}{\sqrt{2\pi}}\begin{bmatrix} 1 \\ e^{-i(k+1)x_2} \\ \vdots \\ e^{-i(k+1)x_n} \end{bmatrix}, \]
which is equivalent to the Vandermonde system
\[ \begin{bmatrix} 1 & 1 & \ldots & 1 \\ 1 & e^{ix_2} & \ldots & e^{i(n-1)x_2} \\ \vdots & \vdots & & \vdots \\ 1 & e^{ix_n} & \ldots & e^{i(n-1)x_n} \end{bmatrix} \mathbf{u} = \begin{bmatrix} 1 \\ e^{-ix_2} \\ \vdots \\ e^{-ix_n} \end{bmatrix}. \]
Setting $z_j = e^{i x_j}$, then the entries of $\mathbf{u}$ are the coefficients of a degree $n-1$ complex polynomial $p(z)$ such that $p(z_j) = 1/z_j$. Setting $q(z) = z p(z)$, it suffices to find a degree $n$ polynomial $q(z)$ such that $q(0) = 0$ and $q(z_j) = 1$. Such a $q$ is unique and
\[ q(z) = 1 - \prod_{j=1}^n \left(1 - z/z_j \right). \]
Then the leading coefficient of $q$, which is also the leading coefficient of $p$, has absolute value $1$, and hence $\|\mathbf{u}\|_2 \geq 1$. 
\end{proof}
\begin{remark}
The trapezoidal rule \emph{uniquely} generates a minimal bilinear quadrature, since in that case $q(z) = \alpha z^n$ for some $|\alpha| = 1$. If $x_j$'s are not equispaced, $q(z)$ has some nonzero lower-order coefficients.
\end{remark}

\subsection{Lobatto quadrature and the non-invertible case}

While minimizing the number of evaluation points will reduce the cost of evaluating a quadrature, it may be advantageous to use more points than is optimal in order to improve accuracy. One example is Lobatto quadratures for polynomials of one variable, which use more points than Gaussian quadratures. In this section, we observe Lobatto quadratures are a special case of a bilinear quadrature where extra evaluation points are used, in which case the matrix function $F(\mathbf{x})$ is non-invertible. The formulation of Lobatto-like bilinear quadratures on general domains is given.

\begin{definition}
Let $\mathbb{P}_{2n-1}$ be $\mu$-integrable on an interval $I = [a,b]$. Then the corresponding \emph{Lobatto quadrature} is a classical quadrature of order $n+1$ exact on $\mathbb{P}_{2n-1}$ with respect to $\mu$ such that if $x_0, \ldots, x_n$ are the nodes, then $x_0 = a$ and $x_n = b$.
\end{definition}

\begin{theorem}
Suppose $\mathbb{P}_{2n-1}$ is $\mu$-integrable on $I$, and let $\{\phi_k\}$ denote the unique set of orthonormal polynomials such that $\deg(\phi_k) = k$. Then there exists a unique Lobatto quadrature of order $n+1$, and the interior nodes $\{x_i: 1 \le i \le n-1\}$ are the zeros of $\tfrac{d}{dx} \phi_{n}(x)$.
\end{theorem}

A Lobatto quadrature of order $n+1$ corresponds to a symmetric bilinear quadrature that is exact on $\mathbb{P}_{n-1} \times \mathbb{P}_{n-1}$ and minimal on $\mathbb{P}_{n}$ in which the $W$ matrix is diagonal and the matrix $F(\mathbf{x})$ is given by
\begin{equation*}
F(\mathbf{x}) = \begin{bmatrix} \phi_1(a) & \ldots & \phi_k(a) \\ \phi_1(x_1) &  & \phi_k(x_1) \\ \vdots & & \vdots \\ \phi_1(x_{n-1}) & & \phi_k(x_{n-1}) \\ \phi_1(b) & \ldots & \phi_k(b) \end{bmatrix}. 
\end{equation*} 

Unlike in the Gaussian quadrature case, the matrix $F = F(\mathbf{x})$ is not square, so there exists infinitely many matrices $W$ satisfying \eqref{eqn.Wmatrixconstraint}. Therefore, the simplified minimization condition \eqref{eqn.mininnerproductsymmetric} cannot be employed, and one must optimize over both the quadrature nodes $\mathbf{x}$ and matrices $W$. In general, suppose $F$ is $m \times k$ and $G$ is $n \times k$, both with full column rank. Then all matrices $W$ satisfying \eqref{eqn.Wmatrixconstraint} are of the form
\begin{equation}
\label{eqn.Wmatrixformula}
W = (F^{\ast})^{+} MG^{+} + Y - F F^{+} Y G G^{+},
\end{equation}
where $Y$ is an arbitrary $m \times n$ matrix. In the symmetric case $F = G$ and $M = I$, a minimal bilinear quadrature is found through the unconstrained minimization
\begin{equation}
\label{eqn.minquad2}
\text{Find } Y \in \R^{m \times m} \text{ and } \mathbf{x} \text{ minimizing } \sigma_1 \left( F^{+} \Gamma + F^{\ast} Y (I - F F^{+}) \Gamma \right)
\end{equation}
While computationally more expensive, this optimization procedure can be used to compute symmetric Lobatto-like bilinear quadratures. First fix points $\mathbf{x}_0$ that the bilinear quadrature is required to use, then construct the (typically non-square) matrix function $F(\mathbf{x}_0, \mathbf{x})$, where only the points $\mathbf{x}$ are varying. Then minimize according to \eqref{eqn.minquad2}. This procedure is applicable for arbitrary domains $\Omega$, any space of continuous functions, and any inner product on that space.

\subsection{Change of variables}

For a bilinear quadrature computing an $L^2(\Omega)$ inner products, a bilinear quadrature can be cheaply constructed for $L^2(\Phi(\Omega))$ inner products, where $\Phi$ is an affine invertible change of variables. For continuous functions $f_i$ on $\Omega$, set
\[
\tilde{f}_i(\Phi(x)) = f_i(x) |\det(D\Phi)|^{-1/2}.
\]
Then
\[
\langle \tilde{f}_i,\tilde{f}_j \rangle_{L^2(\Phi(\Omega))} = \int_{\Phi(\Omega)} \tilde{f}_i(y) \tilde{f}_j(y) \, dy = \int_{\Omega} f_i(x) f_j(x) \, dx = \langle f_i,f_j \rangle_{L^2(\Omega)}.
\]
The Jacobian $D\Phi$ is constant when $\Phi$ is affine, so if the bilinear quadrature on $L^2(\Omega)$ exact on $\F_0 \times \G_0$ is
\[
Q(f,g) = f(\mathbf{x})^{\ast} W g(\mathbf{y)},
\]
a bilinear quadrature on $L^2(\Phi(\Omega))$ for $(\F_0 \circ \Phi^{-1}) \times (\G_0 \circ \Phi^{-1})$ is given by
\begin{equation}
\label{eqn.changeofvariables}
\tilde{Q}(\tilde{f}, \tilde{g}) = \tilde{f}(\Phi(\mathbf{x}))^{\ast} \tilde{W} \tilde{g}(\Phi(\mathbf{y})), \quad \tilde{W} = W |\det(D\Phi)|^{-1}.
\end{equation}

For an $H^s$ inner product with $s > 0$, in general a new bilinear quadrature cannot be cheaply constructed under a change of variables. However, when $\Phi(x) = \lambda U x + b$ is affine with $\lambda \in \R$ and $U$ a unitary matrix, a change of variables can still be performed at low cost. Let $W$ be the matrix in a bilinear quadrature of form \eqref{eqn.innerproductquad} computing $H^1(\Omega)$ inner products. Then write $W = W_0 + W_1$, where $W_0$ is the matrix for a bilinear quadrature that computes $L^2(\Omega)$ inner products. Then a new bilinear quadrature for $H^1(\Phi(\Omega))$ is formed with matrix
\[
\tilde{W} = |\lambda|^{-1} W_0 + |\lambda|^{-3} W_1
\]
and evaluation points mapped by $\Phi$.

\section{Computation}
\label{sec.computation}

In this section, a basic numerical procedure to produce symmetric bilinear quadrature rules is described. Afterward, some numerical examples of bilinear quadrature rules are presented.

\subsection{Orthogonalization}
For a function space $\F_0$, one may initially have a numerical routine to evaluate (up to machine precision) basis functions $\psi_1, \ldots, \psi_k$ for $\F_0$ that are not orthonormal. Assuming that the inner products $\langle \psi_i, \psi_j \rangle = M_{ij}$ can be computed exactly, $F(\mathbf{x})$ is computed from $\Psi(\mathbf{x})$ and Gram matrix $M$ by
\begin{enumerate}
\item Compute the lower triangular matrix $L$ in the Cholesky factorization $M = LL^{\ast}$.
\item For a given $\mathbf{x}$, perform a lower-triangular solve on the matrix equation $\Psi(\mathbf{x})^{\ast} = L Z$.
\item Set $F(\mathbf{x}) = Z^{\ast}$.
\end{enumerate}
The same procedure can be used to produce an orthonormal basis for the function space $\F_1$ that the bilinear quadrature is minimized against.

\subsection{Nonlinear optimization}

For the invertible symmetric case we have reduced our problem to the minimization problem \eqref{eqn.mininnerproduct}:

\begin{equation*}
\text{Find } \mathbf{x} \text{ minimizing } \sigma_1 \left( F(\mathbf{x})^{-1} \Gamma(\mathbf{x}) \right). 
\end{equation*}
This is a nonlinear optimization problem in $d \cdot k$ variables, where $d$ is the dimension of the integration region $\Omega$ and $k = \dim(\F_0)$. 

The problem of minimizing the largest singular value of a matrix function $A(\mathbf{x})$ is equivalent to minimizing the largest eigenvalue of the symmetric positive semidefinite matrix $A^{\ast} (\mathbf{x}) A(\mathbf{x})$. This type of the eigenvalue optimization problem has been extensively studied in its own right; see \cite{overton} \cite{shapiro}.

Often $F(\mathbf{x})$ and $\Gamma(\mathbf{x})$, but not their derivatives, can be accurately computed. Also, the multiplicity of the largest singular values are generally unknown. Consequently, a quasi-Newton method is ideal for the optimization procedure. The objective function is non-convex and typically has multiple local minima, so the optimization procedure is run with many initial guesses. Furthermore, in the presence of many nearby local minima, after each convergent result, the computed points can be perturbed by a small value $\delta$ and the procedure run again with perturbed points as another initial guess. This is repeated until suitable convergence. While this procedure may be expensive, computing a quadrature is typically a one-time cost, after which the quadrature can be used repeatedly for its applications.

In our numerical experiments, we employ a quasi-Newton method with BFGS updates as implemented as part of Matlab's \texttt{fminunc} routine \cite{broyden, fletcher,goldfarb,shanno}. Since $F(\mathbf{x})^{-1} \Gamma(\mathbf{x})$ is a small, dense matrix, its norm is computed by calculating its full SVD. Up to $10^5$ initial random points uniformly distributed across the domain are used, and the procedure is iterated until convergence in double-precision arithmetic.  

For our numerical implementation we do not reinforce the constraint that the evaluation points $x_i$ remain in the integration domain $\Omega$. While in general the full constrained minimization problem may be necessary, we have empirically observed that it is not necessary for quadratures on polynomials. This can be explained by observing that the orthogonal polynomials grow rapidly outside of $\Omega$; thus points outside the domain are not expected to be good candidates for the solution to the minimization problem.

\begin{remarks}
In the case of polynomials it is possible to accurately compute the gradients of $F(\mathbf{x})$ and $\Gamma(\mathbf{x})$, in which case a quasi-Newton method may be unnecessary. The BFGS method has been chosen since it is robust for different function spaces.
\end{remarks}

\subsection{Bilinear quadratures on triangular domains}

In practical applications one of the most important cases to consider is the $L^2$ product of polynomials on a simplex. For example, in the finite element method one typically solves a two-dimensional PDE locally on polynomials supported on triangular domains. The discretization requires computing a number of inner products. In this section we compute bilinear quadratures that are exact on polynomials on a triangular domain.

Because the space of polynomials is affine-invariant it suffices to find evaluation points for polynomials on a reference triangle. Given a bilinear quadrature on a reference triangle a bilinear quadrature for polynomials on any other triangle can be cheaply obtained using the change of variables formula \eqref{eqn.changeofvariables}. A basis of orthogonal polynomials on the right triangle with vertices $(-1, -1), (-1,1), (1, -1)$ is given by
\begin{equation}
\label{eqn.orthtriangle}
K_{m,n}(x,y) = \left( \frac{1 - v}{2} \right)^m P_m\left( \frac{2x + y + 1}{1 - y} \right) P_n^{2m + 1,0} (y),
\end{equation}
where $P_m$ is the $m$th Legendre polynomial and $P_n^{\alpha, \beta}$ is the $n$th Jacobi polynomial with parameters $\alpha, \beta$. These functions can be computed efficiently and stably as in \cite{xiao1}.

Using this basis, symmetric bilinear quadratures exact for the $L^2$ inner product over this right triangle on $\F_0 = \mathbb{P}_n$ and minimal on $\F_1 = \mathbb{P}_{n+1} \cap \mathbb{P}_n^{\perp}$ were computed. The minimal number of evaluation points were used, in which case the number of points required is
\[
k = \dim(\mathbb{P}_n) = \binom{n+2}{2}.
\]
In Table~\ref{table.trianglequad}, for each computed bilinear quadrature rule, the minimized largest singular value $\sigma = \sigma_1(F(\mathbf{x})^{-1} \Gamma(\mathbf{x}))$ is given, along with the $\infty$-norm condition number of the matrix for the approximate orthogonal projection.

In Figure~\ref{fig.trianglequad}, the evaluation points of two bilinear quadrature rules on the equilateral triangle are shown. Notice that the points possess some symmetries. The expectation that quadrature points for polynomials should have some symmetries has been exploited in the past to reduce the complexity of searching for classical quadratures \cite{xiao1}. In the quasi-Newton method used to solve \eqref{eqn.mininnerproductsymmetric}, however, no symmetry conditions were explicitly enforced.

\begin{table}[t]
\begin{center}
\begin{tabular}{l r r r}
$n$ & $k$ & $\sigma$ & $\kappa_{\infty}(F^*W)$ \\ \hline
0 & 1  & 0.00000 & 1.00000e+0\\
1 & 3  & 0.14507 & 2.82218e+0\\
2 & 6  & 0.30373 & 6.29185e+0\\
3 & 10 & 0.47762 & 1.15455e+1\\
4 & 15 & 0.65817 & 2.03810e+1\\
5 & 21 & 0.78394 & 3.39955e+1\\
6 & 28 & 0.87930 & 4.71065e+1\\
7 & 36 & 0.95305 & 8.48889e+1\\
8 & 45 & 1.05595 & 1.09107e+2
\end{tabular}
\end{center}
\caption{Numerical results for $k$-point bilinear quadratures on $\mathbb{P}_n \times \mathbb{P}_n$ for $L^2$ on the interior of the reference right triangle.}
\label{table.trianglequad}
\end{table}

\begin{figure}
\centering
\includegraphics[width = 0.4\textwidth]{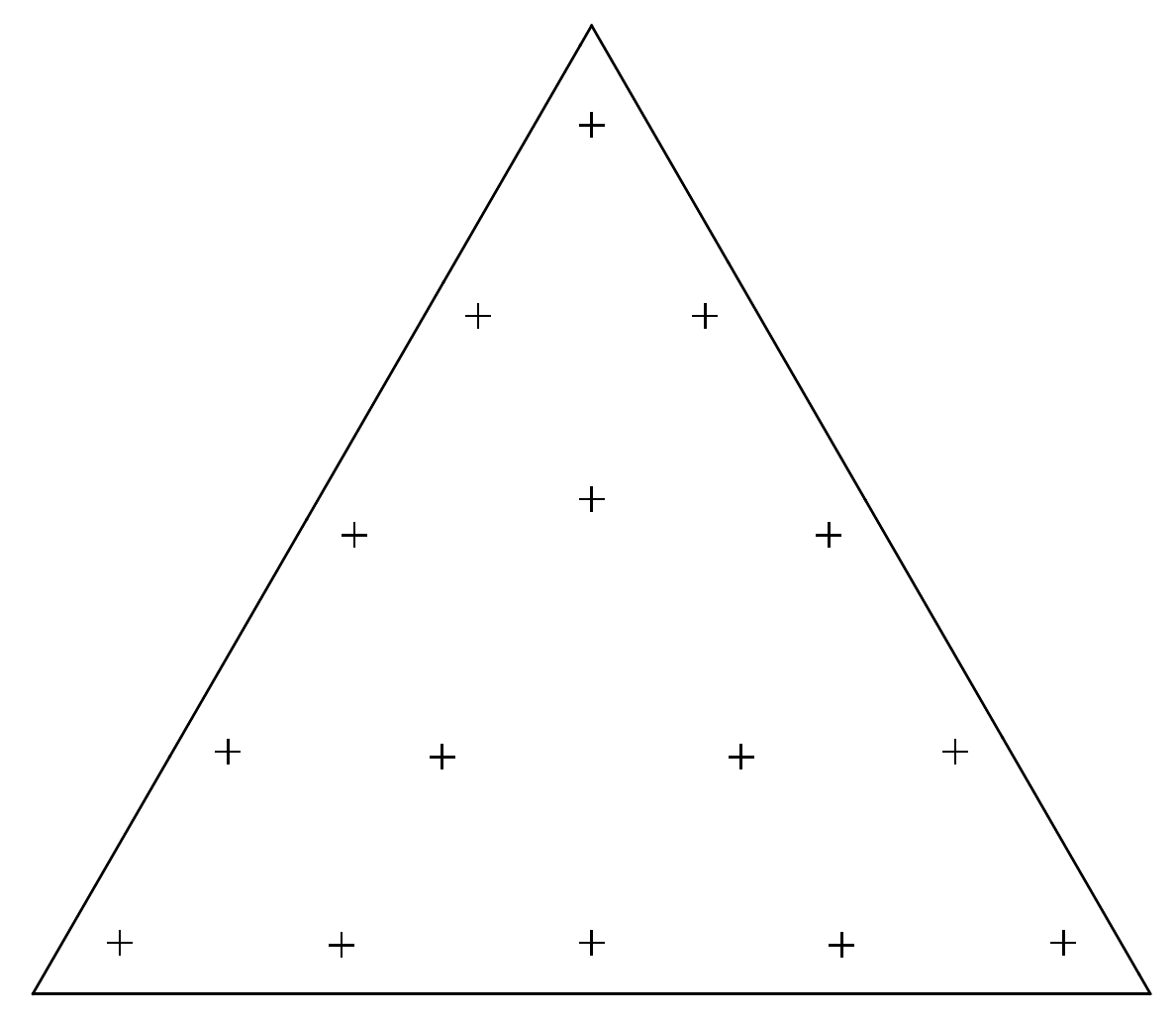}
\includegraphics[width = 0.4\textwidth]{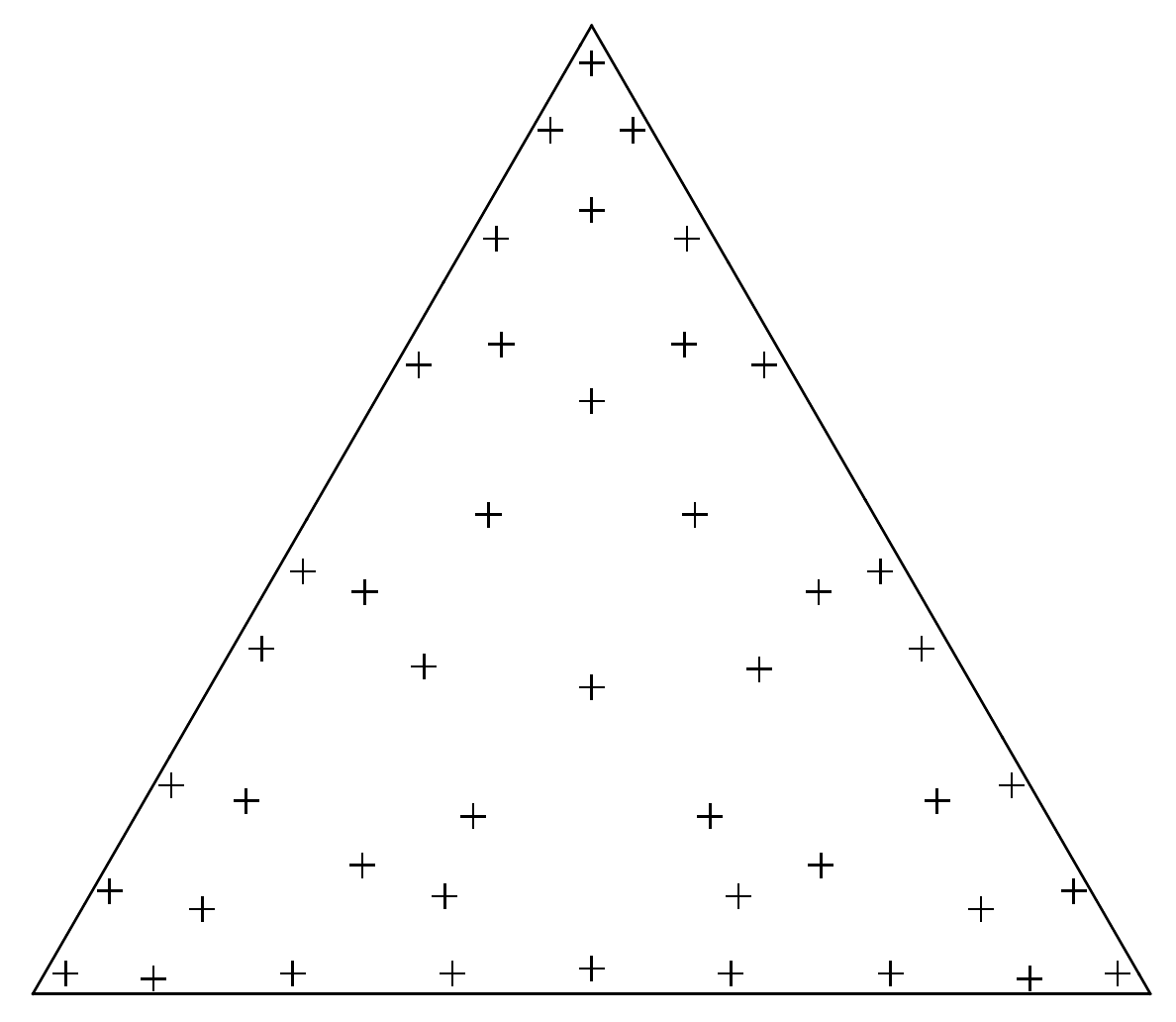}
\caption{Evaluation points for bilinear quadratures on $\mathbb{P}_n \times \mathbb{P}_n$ for $L^2$ on the interior of an equilateral triangle, for $n = 4,8$.}
\label{fig.trianglequad}
\end{figure}

\subsection{Numerical accuracy of quadratures on triangles}

In the section the computed bilinear quadratures on triangles are compared against existing high-order classical quadrature schemes on triangles in the setting of orthogonal projections. Given the space $\F_0 = \mathbb{P}_n$ on $\Omega$ with $L^2$-orthonormal basis $\beta = \{f_i\}$, orthogonal projection operator $P$ onto $\F_0$, and given $g \in C^{\infty}(\Omega)$, we wish to compute
\[
[Pg]_{\beta} = \begin{bmatrix} \langle f_1, g \rangle_{L^2} \\ \vdots \\ \langle f_k, g \rangle_{L^2} \end{bmatrix}.
\]
The column vector $[Pg]_{\beta}$ can be computed using either an approximate orthogonal projection, or a classical quadrature for each entry $\int_{\Omega} f_i g$.

Since the approximate orthogonal projection matrix $F^{\ast} W$, weights of the classical quadrature, and locations of evaluation points are all precomputed, the flop cost for each method is solely determined by the number of evaluation points needed. The 28-point bilinear quadrature as shown in Figure~\ref{fig.trianglequad} was utilized. For comparison we chose two different 28-point classical quadratures, each exact on polynomials of degree up to 11, due to Dunavant \cite{dunavant} and Xiao and Gimbutas \cite{xiao1}, respectively. These quadratures were computed using the libraries available from \cite{burkardt}. Both classical quadratures were similarly transformed to an equilateral triangle of side length 1.

For our numerical experiments, we draw the projected function $g$ from four different probability distributions of functions, which we denote by $\mathbb{P}'_5, \mathbb{P}'_6, C$, and $TP$. 

We define
\[
\mathbb{P}'_n : = \{ g \in \mathbb{P}_n: \|g\|_{L^2} = 1\},
\]
with probability measure given by drawing a random vector of coefficients uniformly in $[-1,1]^k$, and then normalizing the coefficients to have $\ell^2$-norm $1$, and using those as the Fourier coefficients on the orthonormal polynomials on the triangle.

The set $C$ contains smooth functions with slow decay, and is defined by functions of the form
\[
g(x,y) = \frac{1}{1 + (a_1 x + a_2 y)^2},
\]
where $a = (a_1, a_2)$ is drawn uniformly from the unit circle.

The set $TP$ contains smooth non-polynomial functions with oscillations, and has elements of the form
\[
g(x,y) = e^{a_1 x + a_2 y} \cos(4b_1 x + 4 b_2 y) p(x,y),
\]
where parameters $(a_1, a_2)$ and $(b_1, b_2)$ are both drawn uniformly from the unit circle, and $p(x,y)$ is a random element of $\mathbb{P}'_2$ with $L^2$ norm 1 as chosen in the same manner as for the first two cases.

\begin{table}
\begin{center}
\begin{tabular}{ l  | c  c  c  c }
				& $\mathbb{P}'_5$ 	& $\mathbb{P}'_6$ 	& $C$ 				& $TP$ 				\\ \hline
Dunavant 		& 9.38e-14			& 3.97e-01			& 4.97e-05			& 9.06e-03 			\\  
Xiao/Gimbutas 	& 3.29e-15			& 2.73e-01			& 1.91e-05			& 4.74e-03 			\\ 
Bilinear 		& 3.92e-15			& 3.99e-15			& 6.74e-06			& 1.71e-03 			\\ 
\end{tabular}
\end{center}
\caption{Average $\ell^2$-norm relative error in computing approximate orthogonal projection coefficients onto $\mathbb{P}_6$ for four different sets of functions using three methods that use 28 function evaluations.}
\label{table.testprojections}
\end{table}

For each randomly chosen function $g$, we computed the column vector $[P_Q g]_{\beta}$ using the three quadrature methods. The exact value $[Pg]_{\beta}$ was computed with a 295-point classical quadrature that exactly integrates polynomials up to degree 40, as computed in \cite{xiao1}. The $\ell^2$ norm relative error was averaged over $10^4$ randomly generated $g$ for each of the four classes of functions. The resulting average relative errors are shown in Table \ref{table.testprojections}.

On $\mathbb{P}'_5$, all three quadrature rules achieve very high accuracy, with the Dunavant quadrature losing one digit of accuracy and both Xiao/Gimbutas and bilinear quadratures correctly computing the orthogonal projection up to double precision. This is expected since all quadratures are designed to integrate such polynomial functions exactly.

On $\mathbb{P}'_6$, neither classical quadratures are accurate to full precision because both classical quadratures are only capable of exactly integrating polynomials of degree up to $11$. Since the bilinear quadrature can exactly integrate $\mathbb{P}_6 \times \mathbb{P}_6$, it has mean error on the order of machine precision.

On the sets $C$ and $TP$, none of the quadratures are accurate to machine precision since none of the functions are polynomials. However, the bilinear quadrature achieves better accuracy than the classical quadratures despite having the same number of evaluation points.

The existing classical quadratures are already very good, integrating non-polynomial functions from $C$ and $TP$ with several digits of accuracy. Additionally, the classical quadrature of Xiao/Gimbutas performs better than the Dunavant quadrature in all four cases. However, the bilinear quadrature was as good or better than the classical quadratures in each case, despite using the same number of evaluations. This result is explained by the fact that bilinear quadratures are specifically designed for the orthogonal projection problem, while classical quadratures are designed for evaluating a linear functional.

\subsection{Bilinear quadratures on other domains}

\begin{table}[t]
\begin{center}
\begin{tabular}{l r r r}
$n$ & $k$ & $\sigma$ & $\kappa_{\infty}(F^*W)$ \\ \hline
0 & 1  & 0.00000 & 1.00000e+0\\
1 & 3  & 0.67739 & 2.91852e+0\\
2 & 6  & 0.79523 & 7.50137e+0\\
3 & 10 & 0.92888 & 1.17526e+1\\
4 & 15 & 0.97590 & 2.68367e+1\\
5 & 21 & 0.99701 & 3.14417e+1\\
6 & 28 & 1.00066 & 6.42937e+1\\
7 & 36 & 1.00711 & 7.34237e+1\\
8 & 45 & 1.00784 & 1.03464e+2
\end{tabular}
\end{center}
\caption{Numerical results for $k$-point bilinear quadratures on $\mathbb{P}_n \times \mathbb{P}_n$ for $L^2$ on the interior of a square.}
\label{table.squarequad}
\end{table}

\begin{figure}
\centering
\includegraphics[width = 0.4\textwidth]{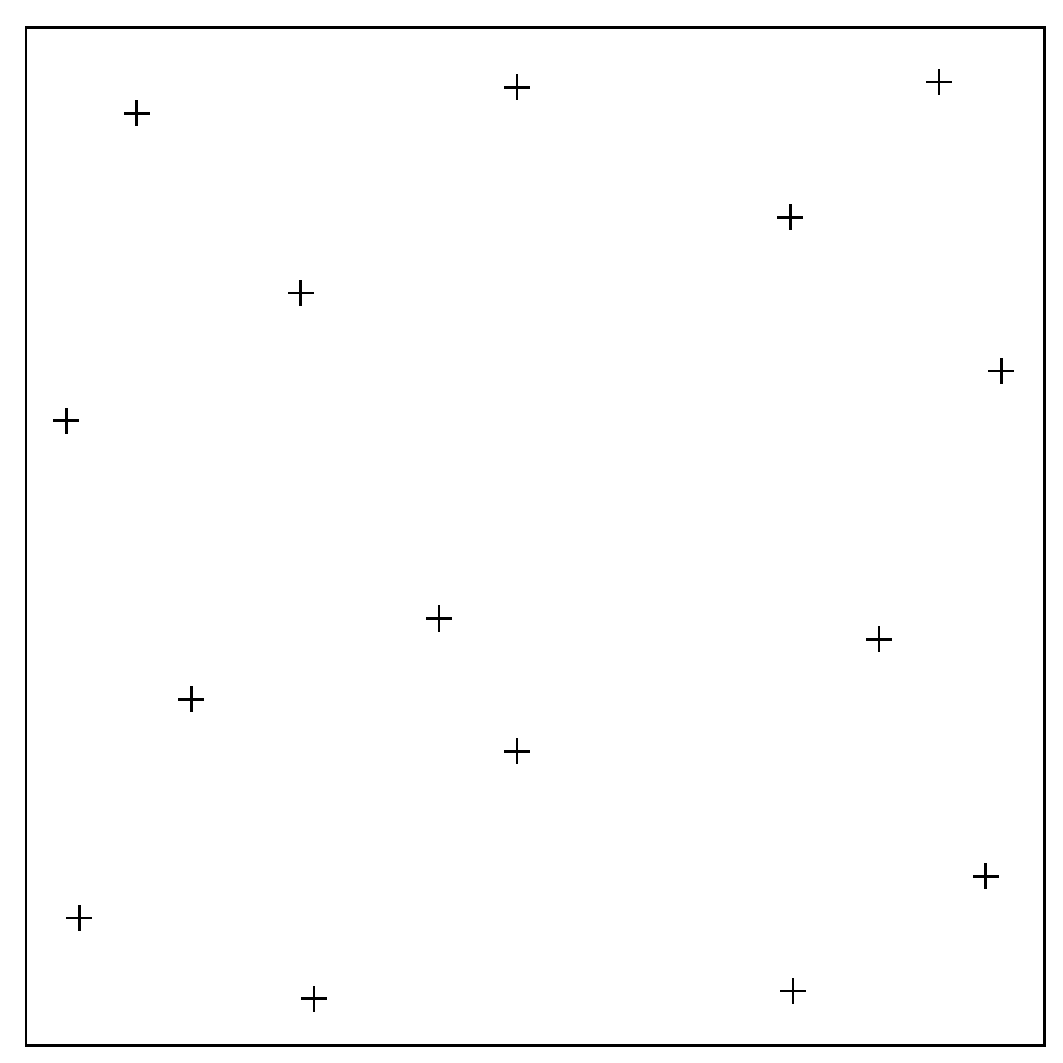}
\includegraphics[width = 0.4\textwidth]{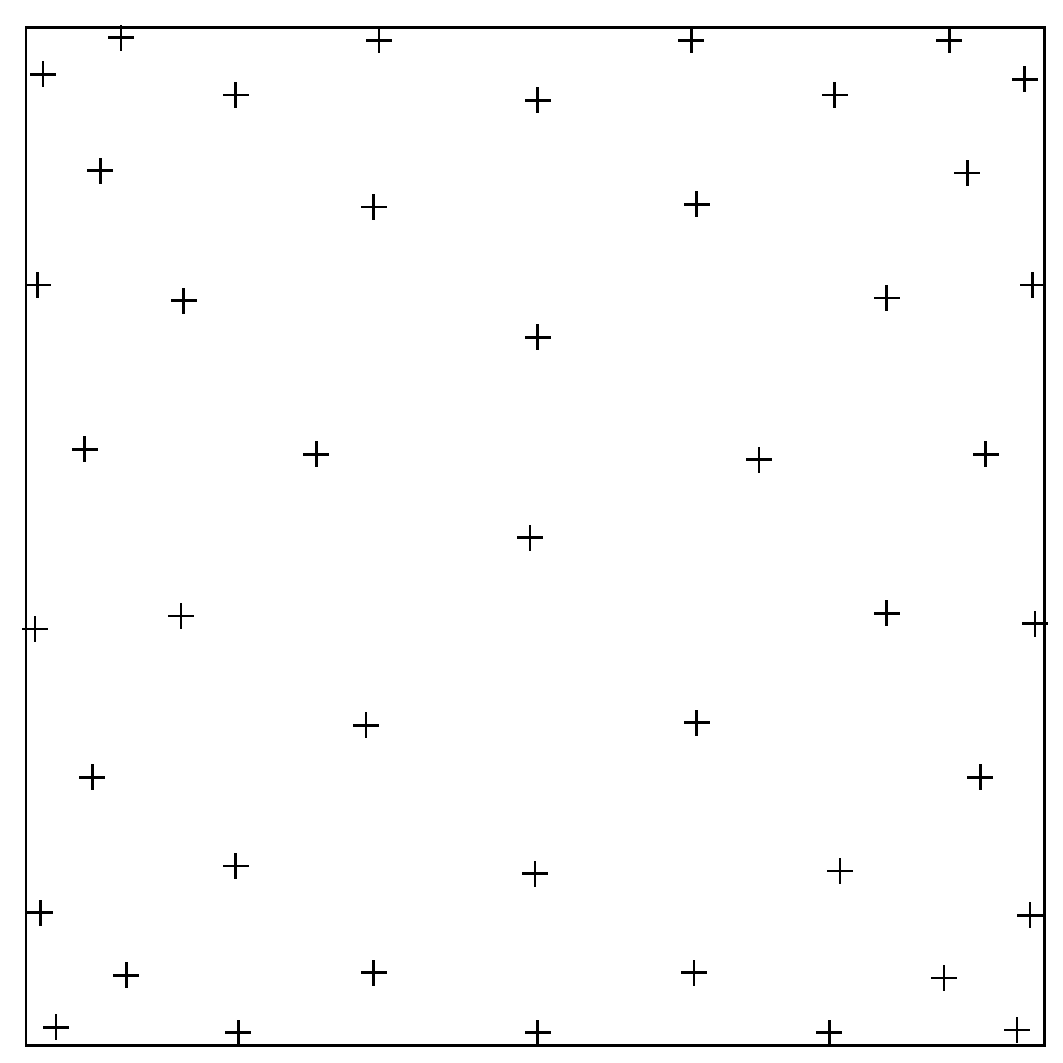}
\caption{Evaluation points for bilinear quadratures on $\mathbb{P}_n \times \mathbb{P}_n$ for $L^2$ on the interior of a square, for $n = 4,8$.}
\label{fig.squarequad}
\end{figure}

In this section bilinear quadratures for $L^2$ inner products of polynomials on the interiors of a square and a circle are computed. We observe that, just as in the case of triangles, minimizing according to \eqref{eqn.mininnerproductsymmetric} produces well-behaved evaluation points.

For the case of the square domain $[-1,1]^2$, orthogonal polynomials are $P_n(x) P_m(y)$, where $P_n$ is the $n$th Legendre polynomial. Table \ref{table.squarequad} shows the minimized leading singular value $\sigma$ and matrix condition number $\kappa_{\infty}$ for several $k$-point bilinear quadratures on the square. Interestingly, the evaluation points on the square do not appear to obey any symmetries.

\begin{remark}
One can produce a classical quadrature scheme on the square by simply taking the tensor product of two Gaussian quadratures on an interval. However, this exactly integrates basis functions of the form $x^{\alpha} y^{\beta}$ with $0 \le \alpha \le n$, $0 \le \beta \le n$, rather than integrating polynomials whose \textit{total degree} does not exceed some value.
\end{remark}

On the unit disk, an orthogonal basis of polynomials is given in polar coordinates by the Zernike polynomials $Z_{m,n}(r, \theta)$, defined by
\begin{align*}
Z_{m,n}(r, \theta) := Q_{m,n}(r) \cos(m\theta), \quad Z_{-m,n}(r, \theta) := Q_{m,n}(r) \sin(m \theta), \\
Q_{m,n}(r) := \sum_{k=0}^{(n-m)/2} (-1)^k \binom{n-k}{k} \binom{n-2k}{\tfrac{n-m}{2} - k} r^{n - 2k},
\end{align*}
where $n \geq m \geq 0$ are integers and $n - m$ is even. Table \ref{table.diskquad} shows the minimized leading singular value $\sigma$ and matrix condition number $\kappa_{\infty}$ for several $k$-point bilinear quadratures on the unit disk.

\begin{table}[t]
\begin{center}
\begin{tabular}{l r r r}
$n$ & $k$ & $\sigma$ & $\kappa_{\infty}(F^*W)$ \\ \hline
0 & 1  & 0.00000 & 1.00000e+0\\
1 & 3  & 0.67617 & 3.04857e+0\\
2 & 6  & 0.79868 & 5.50559e+0\\
3 & 10 & 0.89712 & 1.01509e+1\\
4 & 15 & 0.94133 & 1.59179e+1\\
5 & 21 & 0.97804 & 2.24193e+1\\
6 & 28 & 1.00337 & 3.94055e+1\\
7 & 36 & 1.02908 & 5.60579e+1\\
8 & 45 & 1.07413 & 6.75064e+1
\end{tabular}
\end{center}
\caption{Numerical results for $k$-point bilinear quadratures on $\mathbb{P}_n \times \mathbb{P}_n$ for $L^2$ on the unit disk.}
\label{table.diskquad}
\end{table}

\begin{figure}
\centering
\includegraphics[width = 0.4\textwidth]{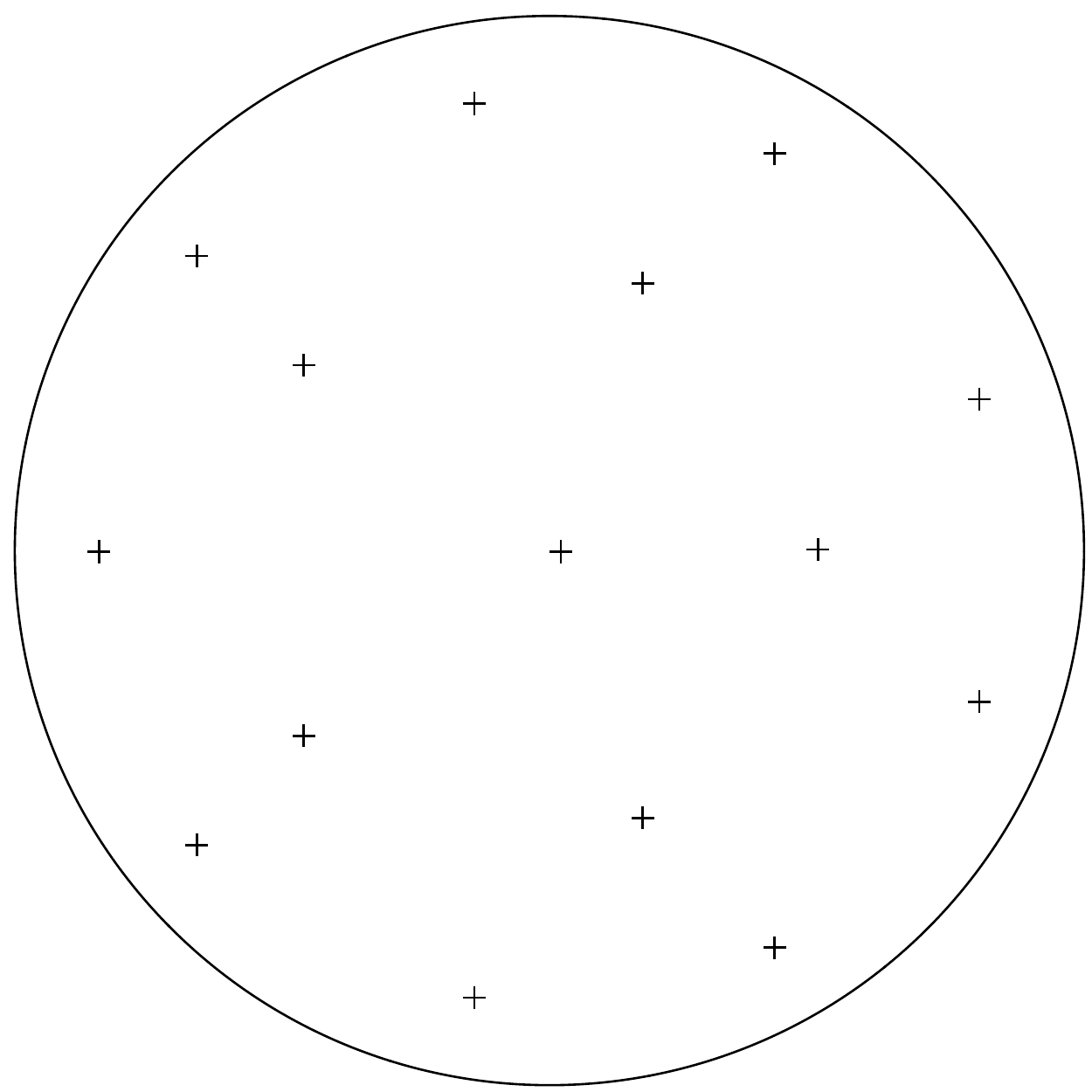}
\includegraphics[width = 0.4\textwidth]{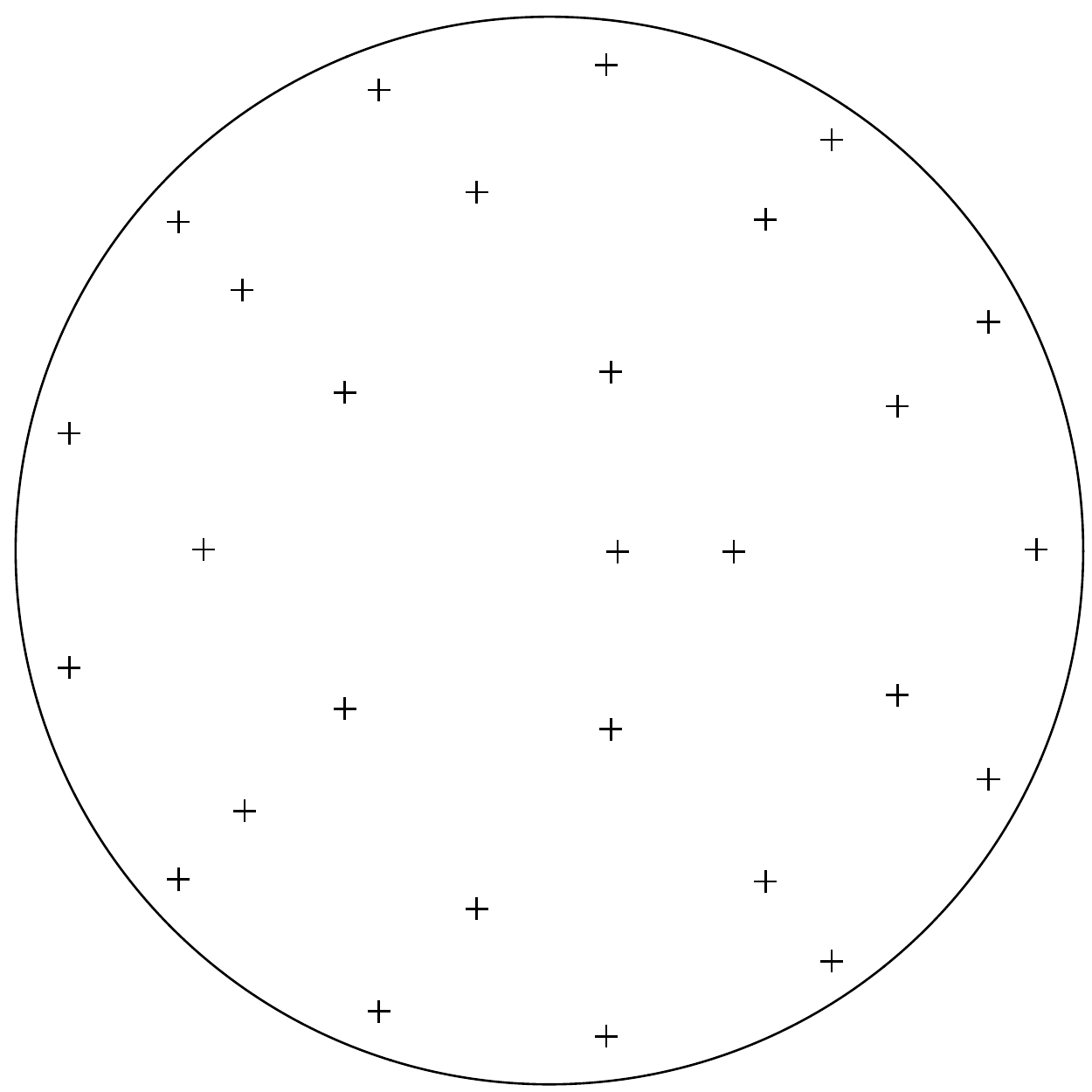}
\caption{Evaluation points for bilinear quadratures on $\mathbb{P}_n \times \mathbb{P}_n$ for $L^2$ on the unit disk, for $n = 4,6$.}
\label{fig.diskquad}
\end{figure}

\subsection{Bilinear quadrature for the Sobolev inner product}

In this section we compute bilinear quadratures that evaluate the Sobolev inner product
\[
\langle f,g \rangle_{H^1} = \int_{\Omega} Df(x) \cdot A(x) Dg(x) + f(x) g(x)\, dx,
\]
where $A(x)$ is symmetric positive definite on $\Omega$. One advantage of a bilinear quadrature for $H^1$ is that the above integral can be numerically evaluated using only point evaluations of $f,g$ and does not require evaluating any derivatives.

For $\Omega = [-1,1]$, bilinear quadratures for $H^1$ on $\mathbb{P}_n \times \mathbb{P}_n$ and minimal on $\mathbb{P}_{n+1} \cap \mathbb{P}_n^{\perp}$ were computed for two positive weight functions $A(x) = 1 + x^2$ and $A(x) = e^{x}$. Orthogonalization was performed by starting with the Legendre polynomials and computing the Gram matrix $M$ using a 40-point classical Gaussian quadrature. 

\begin{table}[t]
\begin{center}
\begin{tabular}{l r r r}
$n$ & $k$ & $\sigma$ & $\kappa_{\infty}(F^*W)$ \\ \hline
1 & 2  & 0.00000 & 5.00000e+0\\
2 & 3  & 0.00000 & 1.38132e+1\\
3 & 4  & 0.00000 & 2.72011e+1\\
4 & 5  & 0.00000 & 6.59254e+1\\
5 & 6  & 0.00000 & 1.21461e+2\\
6 & 7  & 0.00000 & 1.86818e+2\\
7 & 8  & 0.00000 & 2.86549e+2\\
8 & 9  & 0.00000 & 4.22824e+2
\end{tabular}
\end{center}
\caption{Numerical results for bilinear quadratures on $\mathbb{P}_n \times \mathbb{P}_n$ for $H^1[-1,1]$ with the weight function $A(x) = 1 + x^2$.}
\label{table.sobquad1}
\end{table}

\begin{table}[t]
\begin{center}
\begin{tabular}{l r r r}
$n$ & $k$ & $\sigma$ & $\kappa_{\infty}(F^*W)$ \\ \hline
1 & 2  & 0.00000 & 4.52560e+0\\
2 & 3  & 0.00000 & 1.30446e+1\\
3 & 4  & 0.00000 & 2.49183e+1\\
4 & 5  & 0.00000 & 5.13338e+1\\
5 & 6  & 0.00000 & 9.28987e+1\\
6 & 7  & 0.00000 & 1.50063e+2\\
7 & 8  & 0.00000 & 2.28284e+2\\
8 & 9  & 0.00000 & 3.30651e+2
\end{tabular}
\end{center}
\caption{Numerical results for bilinear quadratures on $\mathbb{P}_n \times \mathbb{P}_n$ for $H^1[-1,1]$ with the weight function $A(x) = e^x$.}
\label{table.sobquad2}
\end{table}

In Tables \ref{table.sobquad1} and \ref{table.sobquad2} the singular value $\sigma$ and condition number $\kappa_{\infty}$ are shown for the two bilinear quadratures for $H^1$. In all cases, $\sigma$ is zero up to machine precision, since the exact solution to the minimization \eqref{eqn.mininnerproduct} is the roots of the $(n+1)$th-degree $H^1$-orthogonal polynomial, just as for Gaussian quadratures. We observe that the condition number of the approximation projection matrix $F^{\ast} W$ is larger than in the $L^2$ case. This can be explained by the fact that small perturbations in the function values can lead to large perturbations in the derivatives.

\section{Conclusions}
A quadrature framework for numerically evaluating a continuous bilinear form on function spaces has been presented, and an optimization procedure for computing such quadratures has been outlined. We have argued that this is the correct approach to numerically evaluating orthogonal projections of functions onto a fixed subspace.

We have also observed that the optimization approach for finding bilinear quadratures does not depend on the ambient dimension, the domain of integration, or the function space to be integrated exactly. Despite this generality, in our numerical experiments we found the resulting quadratures perform well, achieving both efficiency and accuracy.

There are several topics to explore in future work. One is the construction and utilization of bilinear quadratures tailored to specific high-order Galerkin methods. Another is the investigation of the performance of bilinear quadratures for evaluating other (non-Sobolev) bilinear forms. Yet another finding an efficient numerical method for solving the optimization problem \eqref{eqn.minquad2} for the non-invertible case. In that case, a bilinear quadrature is not uniquely determined by its evaluation points, and the optimization problem gains many additional degrees of freedom. Lastly, one could investigate the use of bilinear quadratures for solving integral equations. Such quadratures may prove useful in the Nystr\"{o}m discretization of Fredholm integral operators \cite{boland2} or boundary integral equations on domains with corners \cite{bremer2}.

\section{Acknowledgements}
I would like to thank Ming Gu, Benjamin Harrop-Griffiths, Casey Jao, Per-Olof Persson, and John Strain for comments and suggestions. This work was supported by National Science Foundation under grant DMS-0913695 and the Air Force Office of Scientific Research under grant FA9550-11-1-0242.

\nocite{*}
\bibliographystyle{amsalpha}
\bibliography{gaussref}

\end{document}